\documentclass[12pt]{amsart}
\usepackage{amsthm}
\usepackage{comment}
\usepackage{graphicx}
\usepackage{enumerate}

\newtheorem{theorem}{Theorem}[section]

\newtheorem{lemma}[theorem]{Lemma}

\newtheorem{proposition}[theorem]{Proposition}
\newtheorem{remark}[theorem]{Remark}

\newcommand{\BigO}[1]{\ensuremath{\operatorname{O}\left(#1\right)}}
\newcommand{\LitO}[1]{\ensuremath{\operatorname{o}\left(#1\right)}}
\newcommand{\BiO}[2]{\ensuremath{\operatorname{O}^{#1}_{#2}}}
\newcommand{\BiOv}[3]{\ensuremath{\operatorname{O}^{#1, #3}_{#2}}}

\def\R{\mathbb{R}}

\def\v{v}

\def\ft{\tilde{f}}

\def\R{\mathbb{R}}

\def\r{r}

\def\hc{f}

\def\opf1{\mathcal{F}}

\def\co{\mathcal{C}}

\def\qq{q}

\def\w{w}
\def\b{b}

\def\bv{c}

\def\Ot{\widetilde{\Omega}}
\def\om{\tilde{\omega}}

\def\ga{\tilde{g}}
\def\ha{\tilde{h}}

\newcommand\f[1]{\hc_{#1}}

\author{M. Aguareles,  I. Baldom\`a \& T. M-Seara}
\title{On the asymptotic wavenumber of spiral waves in $\lambda-\omega$ systems.}
\date{\today}
\begin{document}
\begin{abstract}
In this paper we consider spiral wave solutions of a general class of $\lambda-\omega$ systems with a small parameter $q$ and we prove that the asymptotic wavenumber of the spirals is a $\mathcal{C}^{\infty}$-flat function of 
the perturbation parameter $q$. 
\end{abstract}
\maketitle

\section{Introduction}
Rigidly rotating spiral waves are commonly found in many chemical systems and biological processes
\cite{kura84, lechleiter91, winfree72, zaikin70}.
In particular they are most likely to occur in oscillatory models having a rotational symmetry, such as generic $\lambda-\omega$ systems
\cite{kura84}, \cite{scheel98}. These can be derived as the normal form of oscillatory reaction-diffusion systems near a Hopf bifurcation and read:
\begin{eqnarray}
u_t &=& \Delta u + \lambda(f) u -\omega(f) w,\label{lo1}\\
w_t &=& \Delta w + \omega(f) u +\lambda(f) w,\label{lo2}
\end{eqnarray}
where  $u=u(x,y,t)$, $v=v(x,y,t)$ and $\Delta$ denotes the Laplacian.
$\lambda$ and $\omega$ are real functions of $f = \sqrt{u^2+w^2}$.
The conditions that $\lambda$ usually satisfies are: $\lambda(1)=0$, to ensure that the system has a space independent limit solution and
$\lambda'<0$, to guarantee that this limit cycle is stable to homogeneous (space independent) perturbations.
As for $\omega$, based on stability considerations (see \cite{kopell}), it is usually assumed that $|\omega'|$ is small.

Numerical computations reveal that the system \eqref{lo1}-\eqref{lo2} exhibits solutions in the shape of $n$-spirals (see for instance \cite{bohr97, garvie2005}) and more precisely, in the shape of Archimedian spiral waves with a specific frequency $\Omega$.
These rigidly rotating solutions of \eqref{lo1}-\eqref{lo2} can then be written like
\begin{equation}
\label{uv}
\begin{split}
&u(r,\phi,t)=f(r) \cos\left(\Omega t + n\phi - \int_0^r v(s)\,ds\right),\\
&w(r,\phi,t)=f(r) \sin\left(\Omega t + n\phi - \int_0^r v(s)\,ds\right),
\end{split}
\end{equation}
being $r$ and $\phi$ the polar radius and azymuthal coordinates of the plane and thus the Laplacian can be expressed as
$\Delta = \partial_{rr}+\partial_{r}/r+\partial_{\phi\phi}/r^2$.
Therefore, since $f(r)$ plays the role of a modulus, $f(r)\geq 0$ $\forall r>0$ and also $f(0)=0$ in order for $u$ and $w$ to be regular at $r=0$. Also, in order for these functions to have the shape of a spiral, the phase must increase or decrease monotonically  as one moves away from the centre of the spiral and so $v(r)$, which is usually denoted as the \emph{local wavenumber}, must have a constant sign for all $r$. In the particular case where $n=0$, the phase is purely radial and they are usually denoted as \emph{target patterns} since the lines of constant phase become concentric rings, that is to say, along any radial line, the pattern is asymptotically that of a plane wave. 

Substituting the particular expressions \eqref{uv} in \eqref{lo1}-\eqref{lo2} one obtains a set of ordinary differential equations in terms of the radial polar variable, $r$, that reads
\begin{eqnarray}
0&=& f''+\frac{f'}{r}-f\frac{n^2}{r^2}+f(\lambda(f)-v^2),\label{edo1}\\
0&=&fv'+\frac{fv}{r}+2f'v + f(\Omega-\omega(f)).
\label{edo2}
\end{eqnarray}
We note that any arbitrary constant can be added to the phase of the sine and cosine functions of $u$ and $w$ in \eqref{uv} and they would still yield the same equations \eqref{edo1}-\eqref{edo2}.

Using the identity
\begin{equation}
\label{propv}
f\v'+\dfrac{fv}{r}+2f'v=\dfrac{(f^2v r)'}{rf},
\end{equation}
along with the fact that $(f^2(r) v(r) r)_{|r=0}=0$, equation \eqref{edo2} can be expressed in the integral form,
\begin{equation}
v(r)=(rf^2(r))^{-1}\int_0^r t f^2(t)(\omega(f(t))-\Omega) \,dt,
\label{v}
\end{equation}
and this yields $v(0)=0$.

Archimedian spiral waves are characterized by the fact that the distance between two neighbouring fronts of the isophase lines tends to a constant, as $r\to\infty$.
That is to say, if we consider two points of an isophase line ($n\phi-\int_0^r v(s)\,ds=ctant$) one with coordinates $(\phi, r)$ and the following one on the same radial line with coordinates $(\phi + 2\pi, r+\delta(r))$, one obtains,
$$n\phi + \int_0^r v(s)\,ds = n(\phi + 2\pi) + \int_0^{r+\delta(r)} v(s)\,ds.$$
The separation between these two fronts is thus here represented by $\delta(r)$ and satisfies
$$\int_r^{r+\delta(r)} v(s)\,ds = 2\pi n.$$
Then, for Archimedian spiral waves it is expected that $\delta(r)\to D<\infty$ as $r\to\infty$. Using the mean value theorem in the last equality gives $v(r)\to v_{\infty}<\infty$ as $\r\to\infty$ with $v_{\infty}=2\pi n/D$, that is to say, $v_{\infty}$ is proportional to the inverse of the spirals' front separation $D$, and it is usually known as the \emph{asymptotic wavenumber}. As for the modulus, $f(r)$, the type of solutions that have been observed are such that $f(r)$ has a bounded limit and $f'(r)\to 0$ as $r\to\infty$. We will therefore focus on solutions of \eqref{edo1}-\eqref{edo2} such that $f(0)=0$, $f(r)$ and $v(r)$ have bounded limits and $f'(r)\to 0$ as $r\to\infty$.
These are four restrictions to a third order system of  differential equations which suggests that there exists a selection mechanism for the frequency $\Omega$, that is to say, $\Omega$ cannot be arbitrary. In fact, Kopell \& Howard in \cite{kopell} establish the existence of spiral wave solutions ($n\neq 0$) and target patterns ($n=0$) for a particular value of $\Omega$ when $\lambda(1)=0$, $\lambda'(x)<0$ and $\omega(x)<0$. These are then solutions of the system \eqref{edo1}-\eqref{edo2} for $n\in\mathbb{Z}$ with precisely the boundary conditions described above, that is,
\begin{equation}
\label{condi}
\begin{split}
&f(0)=v(0)=0,\\
&\lim_{r\to\infty}f(r)=f_{\infty} <\infty,\quad \lim_{r\to\infty} f'(r)=0,\quad \lim_{r\to\infty} v(r)=v_{\infty}<\infty \\
&\textrm{$f(r)> 0$ and $v(r)$ has constant sign $\forall r>0$.}
\end{split}
\end{equation}
Spiral wave solutions of systems of the type in \eqref{edo1}-\eqref{edo2} with some particular functions $\lambda(x)$, $\omega(x)$ have been studied by numerous researchers. For instance, Hagan in \cite{hagan82} considers the particular case of the complex Ginzburg-Landau equation where
$\lambda(x) = 1-x^2$ and $\omega(x) = -q x^2$.
He uses the method of matching formal asymptotic expansions to construct spiral wave solutions for small values of the parameter $q$.
In particular, he formally finds that  the asymptotic wavenumber $v_\infty = \lim _{r\to\infty} v(r)$ and $\Omega-q$ are exponentially small in $q$.
Also,  Greenberg in \cite{Greenberg} uses a formal perturbation technique to construct solutions of \eqref{edo3}-\eqref{edo4} when $\lambda(x) = 1-x$ and $\omega(x) = 1+q (x-1)$.
 Kopell and Howard in \cite{kopell}  establish the existence of spiral wave solutions of \eqref{edo1}-\eqref{edo2} under the hypothesis that
 $\lambda(1)=0$, $\lambda'(\cdot)<0$, $\omega'(\cdot)<0$ and $|\omega'(\cdot)|<<1$.

The ultimate motivation of the work in this paper is precisely to investigate the exponentially small character of the asymptotic wavenumber for a general class of $\lambda-\omega$ systems. 
In particular, in this paper we consider a general class of $\lambda-\omega$ systems with the following conditions:
\begin{itemize}
\item[(A1)] \label{A1}
$\lambda$ and $\omega$ belong to $\co^{\infty}(\R)$ and they are such that $\lambda(1) = 0$, $\lambda'(1) <0$, and $|\omega'|<<1$.
We remark that by suitably rescaling the radius variable $r$ and the phase function $v$ a new function $\bar{\lambda}$ may be written such that
$\bar{\lambda}(0)$ has any prescribed value.
Therefore, and without lost of generality we also assume that $\lambda(0)=1$.\\
\item[(A2)] \label{A2}
$x\lambda(x)$ is concave, that is to say, $\partial^2_x(x\lambda(x))<0$.
\end{itemize}

Since we assume that $|\omega'|<< 1$,  we shall write $\omega (x)= \omega_0+\qq\bar{\omega}(x)$, for $0\le\qq<<1$.
We introduce a new parameter $\bar \Omega$ such that the frequency may also be written like $\Omega=\omega_0+\qq\bar{\Omega}$.
Dropping the bars to simplify the notation
 equations \eqref{edo1}-\eqref{edo2} read
\begin{eqnarray}
0&=& f''+\frac{f'}{r}-f\frac{n^2}{r^2}+f(\lambda(f)-v^2),\label{edo3}\\
0&=&fv'+\frac{fv}{r}+2f'v + \qq f(\Omega-\omega(f)),
\label{edo4}
\end{eqnarray}
where $0< \qq <<1$ and $\Omega \in \R$ are the new parameters. As for the boundary conditions, we also consider the ones in \cite{kopell} given in \eqref{condi}.

In this paper we prove that, for these solutions to exist, a necessary condition is that
$\Omega=\Omega(q)$ has to be a $\mathcal{C}^\infty$ function of $q$ such that $\partial_q^k\Omega(0)=0$, for $k\geq 1$.
To prove this result  we
provide a formal expansion in $\qq$
of the solutions of equations \eqref{edo3}-\eqref{edo4} with boundary conditions \eqref{condi}.
We obtain an infinite set of differential equations with suitable boundary conditions, one for each order in $\qq$, and we rigorously prove that all these equations have a unique bounded solution if and only if $\Omega(q)-\Omega(0)$ is a $\mathcal{C}^\infty$-flat function of $\qq$.
Therefore, equations \eqref{edo3}-\eqref{edo4} with boundary conditions \eqref{condi} can be solved up to any order in $\qq$.
As a straightforward consequence of our results
the solutions obtained in \cite{kopell} satisfy that $\Omega-\Omega(0)$ and  $v_{\infty} =\lim_{r\to\infty} v(r)$ are $C^{\infty}$-flat functions of $\qq$.
This is then a beyond all orders phenomenon and the rigorous study of the asymptotic values of $\Omega$ and  $v_{\infty}$ will be the goal of a forthcoming paper.

The paper is organised as follows.
We start in Section \S\ref{mainT} by posing the particular shape of the formal solution of \eqref{edo3}-\eqref{edo4}.
We then introduce our main theorem \ref{main} of existence and uniqueness of this formal solution provided the power series for $|\Omega-\Omega(0)|$ has vanishing terms. We also provide a set of numerical computations for the classical complex Ginzburg-Landau problem where $\lambda(x)=1-x^2$ and $\omega(x)=-x^2$, which suggests that $\Omega-\Omega(0)$ is indeed not zero but exponentially small in $q$. In Section \S\ref{power} we prove the main result as follows: we start by obtaining the equations and boundary conditions that each term in the asymptotic expansion satisfies and we then proceed by induction.
\section{Main result: a formal solution}
\label{mainT}
In this section we introduce and justify the expected particular form of the asymptotic expansion in $q$ for the solution of system
\eqref{edo3}-\eqref{edo4} with boundary conditions \eqref{condi}.
We first start with a technical lemma that we shall prove in Section \ref{subsection:omega} below.

\begin{lemma}\label{lemma:omega}
The system \eqref{edo3}-\eqref{edo4} with boundary conditions \eqref{condi} has a solution $(f(r),v(r))$ if and only if
\begin{equation}\label{derfvtofv}
v_{\infty}^2-\lambda(f_{\infty})=0, \quad  \omega(f_{\infty}) -\Omega=0.
\end{equation}
In addition,
$
\lim_{r\to \infty} v'(r)=0.
$

As a consequence, in order to have solutions of \eqref{edo3}-\eqref{edo4}-\eqref{condi},
the parameter $\Omega$ has to be a suitable function of $q$, i.e: $\Omega=\Omega(q)$.
\end{lemma}
This theorem thus states that $v_\infty$ and $f_\infty$ are both functions of only the parameter $q$. To avoid cumbersome notation, we shall in general omit the dependence of $q$ unless such omission leads to error.

Lemma \ref{lemma:omega} above implies that the solution of system \eqref{edo3}-\eqref{edo4} we are dealing with only depends on the
small parameter $q$. We will call it $(f(r;\qq),v(r,\qq))$.
By inspecting equations \eqref{edo3} and \eqref{edo4}, one observes that the modulus $f(r;\qq)$, as well as the unknown frequency
$\Omega(\qq)$, are even functions of $\qq$, that is $\f{}(r;\qq)=f(r;-\qq)$ and
$\Omega(\qq)=\Omega(-\qq)$, while $v$ is an odd function of $\qq$, and so
$v(r;\qq)=-v(r;-\qq)$.
We can thus restrict to positive values of $\qq$ without lost of generality. Moreover, using this even and odd character of the functions with respect to $\qq$ we shall formally find the solutions to \eqref{edo3}-\eqref{edo4} as power series in $\qq$ of the form:
\begin{equation}
\begin{split}
\label{exp}
&f(r;q) = \sum_{i\geq 0} \f{i}(r) \qq^{2i},\quad v(r;q) = \qq\sum_{i\geq 0} \v_{i}(r)\qq^{2i},\\
&\Omega(q)=\sum_{i\geq 0} \Omega_{i}\qq^{2i}.
\end{split}
\end{equation}
In what follows we will find the differential equations and the boundary conditions that  $\f{k}(r)$ and $\v_{k}(r)$
have to satisfy. In order to solve these differential equations, we
will find  the terms for the expansion of the frequency, $\Omega_k$.

Since we will deal with the behaviour as $r\to 0$ and $r \to +\infty$, we also introduce the notation
\begin{equation}\label{notation}
\psi =\BiO{l}{m} \Longleftrightarrow \psi(r) = \BigO{r^m},\; r \to 0 \ \mbox{and}  \ \psi(r) = \BigO{r^{-l}}\; r\to +\infty,
\end{equation}
and
\begin{equation}\label{notationv}
\psi=\BiOv{l}{m}{j} \Longleftrightarrow \psi(r) = \BigO{r^m},\; r \to 0 \ \mbox{and}  \  \psi(r) = \BigO{\log (r)^j r^{-l}}\; r\to +\infty,
\end{equation}
which will be used along this paper without special mention.

The main result in this paper is:
\begin{theorem}
\label{main}
Assume hypotheses (A1)-(A2) hold. Then the system \eqref{edo3}-\eqref{edo4}-\eqref{condi} has a unique formal solution of the form
\eqref{exp} with $\lim_{r\to +\infty}f_0(r)=1$ and satisfying that for all $k\geq 0$:
$$
\f{k}(0)=\v_{k}(0)=0 \ \mbox{and} \ \f{k}(r),  \v_{k}(r)  \ \mbox{are bounded as} \ r\to\infty
$$
if and only if
$$
\Omega_0=\omega(1), \ \mbox{ and } \ \Omega_k=0, \ \forall k\geq 1.
$$
The functions $\f{k}(r), \v_{k}(r)$ also satisfy that
$$
\lim_{r\to\infty} \v'_{0}(r)=0, \quad \lim_{r\to\infty} \v_{0}(r)=0,
$$
and
$$
\lim_{r\to\infty} \f{k}(r)=\lim_{r\to\infty} v_k(r)=0,\quad\textrm{for all $k>0$}.
$$

Moreover,
\begin{align*}
&\f{0}(r) = \BigO{r^{n}} \;\; \text{as}\;\;r\to 0, \;\;1-\f{0}(r)  = \BigO{r^{-2}}\;\;  \text{as}\;\;r\to +\infty. \\
&\f{0}'\in \BiO{3}{n-1},\;\; \f{0}''\in \BiO{4}{\max\{n-2,0\}}\\
&\v_0 \in \BiOv{1}{1}{1},\;\; \v_0'\in \BiOv{2}{0}{1},\;\; \v_0''\in \BiOv{3}{0}{1}
\end{align*}
and for $k\geq 1$,
\begin{align*}
&\f{k} \in \BiOv{2}{n}{2k},\;\; \f{k}' \in \BiOv{3}{n-1}{2k},\;\;\f{k}''\in \BiO{3}{\max\{n-2,0\}}\\
&\v_{k} \in \BiOv{1}{1}{2k+1},\;\; v_{k}'\in \BiOv{2}{0}{2k+1},\;\; v_{k}''\in \BiOv{3}{0}{2k+1}.
\end{align*}

Finally if $\omega$ is a monotone function, $v_0$ has constant sign.
\end{theorem}
\begin{remark}
\label{Remq0}
Note that when $\qq=0$, equation \eqref{edo4} becomes
$$
0=fv'+\frac{fv}{r}+2f'v=\frac{(f^2 v r)'}{f r},
$$
so $f^2(r;0) v(r;0) r \equiv c$, which, upon evaluating at $r=0$ gives $c=0$.
Since we are interested in non trivial solutions for $f(r;0)$, we obtain that $v(r;0)\equiv 0$. As a consequence,
$v_{\infty}=0$. Henceforth, equations \eqref{derfvtofv} imply that $\lambda(f_{\infty})=0$ and $\Omega=\omega(f_{\infty})$. Since by assumption
(A1) $\lambda(1)=0$ it seems natural to choose $f_{\infty}=1$. In fact, if $\lambda'(x)<0$ this is the only choice.

Moreover, since $\lambda'(1)<0$, when $q\neq 0$, for any $v_{\infty}=O(q)$, the equation
$\lambda(f_{\infty})=v_{\infty}^2$ has a solution $|f_{\infty}-1|=O(q)$, if $|q|$ is small enough.

On the other hand, if $\lambda(x)$ has another zero $x_0<1$ satisfying
$\lambda'(x_0)<0$, Theorem \ref{main} can also be applied in this case by rescaling $f=\overline{f}x_0$.
In conclusion, the condition $\lim_{r\to +\infty}f_0(r)=1$ is not restrictive.
\end{remark}

From this theorem we conclude that, if system \eqref{edo3}-\eqref{edo4}-\eqref{condi} has solution $(f(r;q),v(r;q))$
either $v_{\infty}=\lim _{r\to \infty}v(r)=0$
or $v_{\infty}=\lim _{r\to \infty}v(r)=O(\qq ^k)$, $\forall k\ge 0$. That is, the solution $v(r;q)$ either vanishes as $r\to\infty$ or it is $C^{\infty}$-flat  in $\qq$.
In fact, numerical computations reveal that $v(r;q)$ is indeed not zero at infinity. 

\subsection{Numeric computations}
As an example we have considered a Ginzburg-Landau system for $n=1$, which corresponds to $\lambda(x)=1-x^2$ and $\omega(x)=-x^2$, that is:
\begin{align}
0&= f''+\frac{f'}{r}-f\frac{1}{r^2}+f(1-f^2-v^2),\label{CGLn1}\\
0&=fv'+\frac{fv}{r}+2f'v + qf(\Omega-f^2).\label{CGLn2}
\end{align}
with boundary conditions \eqref{condi}. We have used a MATLAB routine to obtain $v_\infty$ which uses a finite difference scheme implementing the three-stage Lobatto Illa formula which provides a $C^1$-continuous solution that is fourth-order accurate uniformly in the interval of integration. 

In \cite{hagan82} and \cite{AgChW10} it is formally obtained an expression for $v_\infty$ as a function of $q$, which is found to be of the form $v_\infty \sim Ae^{-B/\qq}\qq^{-1}$. Performing a linear fit of $\log(\qq v_\infty)$ with 95\% confidence we obtain $B=1.588191499224517$, using moderate values of $\qq\in[0.2, 0.5]$, so it seems to agree with the predicted value $B=\pi/2$. A rigorous computation of $A$ and $B$ would require working with multiprecision and it is beyond the scope of this paper. 

\begin{figure}
\begin{center}
\includegraphics[scale=0.4]{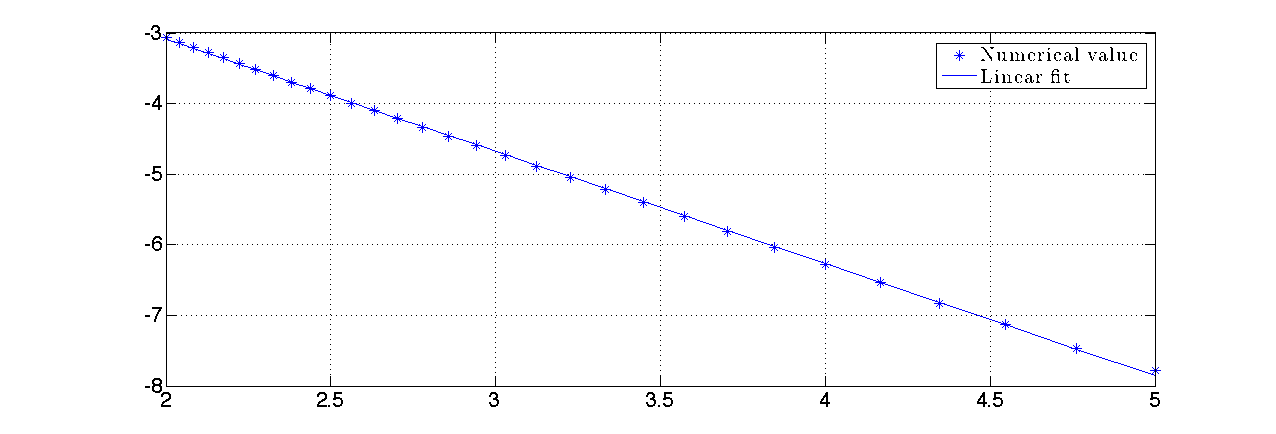}
\end{center}
\caption{Value of $\log(\qq v_{\infty})$ as a function of $1/\qq$, with $\qq\in[0.2,0.5]$ for the system \eqref{CGLn1}-\eqref{CGLn2}.}
\end{figure}



\subsection{Dependence of $\Omega$ on the parameter $q$}\label{subsection:omega}
We now prove Lemma \ref{lemma:omega}.
Assume that system $\eqref{edo3}-\eqref{edo4}-\eqref{condi}$ has a solution $(f(r;\qq,\Omega),v(r;\qq,\Omega))$. We will omit the dependence on
the parameters if there is no danger of confusion.
To prove Lemma \ref{lemma:omega} we first recall that when $\qq=0$, $v\equiv 0$, as it is pointed out in the Remark \ref{Remq0}. As a consequence, $v_{\infty}=O(q)$.

Now we check that $f_{\infty}\neq 0$. Assume that $f_{\infty}=0$. In this case, since $v_{\infty}=O(q)$, we have that
$\v_{\infty}<<1=\lambda(0) = \lambda(f_{\infty})$ taking $q$ small enough. It follows that,
for $r_0$ large enough
$ n^2/r^2 < \lambda(f(r))-v^2(r)$, for all $r\geq r_0$ and using that $f$ is a positive solution of \eqref{edo3}:
$$
(r \f{}'(r) )' = r \f{}''(r) + \f{}'(r)<0 \Longrightarrow
\f{}'(r) < \frac{r_{0}}{r}\f{}'(r_0) \Longrightarrow
\f{}(r) - \f{}(r_0) < r_0 \f{}'(r_0) \log (r_0^{-1} r).
$$
Since $\f{\infty}=0$ and $f(r)>0$, one can take $r_0$ large enough such that
$f'(r)<0$ for $r\geq r_0$. Taking $r\to \infty$ in  the above inequality, we find a contradiction with the fact that $f$ is a bounded function.

According to equation \eqref{edo3}, using that $f'(r)\to 0$ as $r\to +\infty$, we obtain that
$\lim_{r\to\infty} f''(r) = -f_{\infty}(\lambda(f_{\infty})-v_{\infty}^2),$
while equation \eqref{edo4} provides that $\lim_{r\to\infty} v'(r)f_\infty= -\qq f_{\infty} (\Omega -\omega(f_\infty))$.
Then, taking into account
that $f_{\infty} \neq 0$ and
the simple fact that:
\begin{equation}\label{simplefact}
\text{if } h(r)  \text{  is bounded for all $r>0$ and  }\lim_{r\to +\infty} h'(r)=b, \text{  then  }b=0,
\end{equation}
which is immediate by H\^opital's rule, it is found that $\lim_{r\to +\infty}f''(r)= 0$  and  $\lim_{r\to +\infty}v'(r)= 0$.
Therefore, one is left with the couple of equations for the boundary values at infinity:
$$
\lambda(f_{\infty})-v_{\infty}^2=0,\quad \Omega -\omega(f_\infty)=0.
$$

Finally, if we explicitly write the dependence on $\qq,\Omega$ and we have that, to have solutions
of our problem:
$$
\chi(\qq,\Omega):=\Omega-\omega(f_{\infty}(q,\Omega))=0.
$$
Note that, when $\qq=0$, for any value of $\Omega$, $v(r;0,\Omega)=0$ and $f$ has to satisfy the equation
$$
f'' + \frac{f'}{r} -f \frac{n^2}{r^2} + f\lambda(f)=0
$$
which is independent of $\Omega$ so that $f_{\infty}(0,\Omega)=f_{\infty}^0$ does not depend on $\Omega$ and thus $\partial_{\Omega}f_{\infty}(0,\Omega)=0$. Therefore, using this and differentiating $\chi(\qq,\Omega)$ with respect to $\Omega$ one is left with  $\partial_{\Omega}\chi(0,\Omega(f_{\infty}^0))=1\neq 0$, which along with the fact that $\chi(0,\Omega(f_{\infty}^0))=0$, the implicit function theorem defines
a function $\Omega(q)$ such that $\chi(\qq,\Omega(q))=0$ if $|\qq|$ is small enough.

\section{Power series expansions: proof of Theorem \ref{main}}
\label{power}
The idea of the proof is as follows: we first start by describing the system of equations for $\f{k},\v_{k}$, $k\geq 0$ introduced in \eqref{exp}. We then deduce the boundary conditions that $\f{k},\v_{k}$ must satisfy in order to be bounded solutions of such equations. 
We then prove the existence of $\f{k},\v_{k}$, along with some useful properties of the leading order terms $\f{0},\v_0$.

We emphasize that this is a formal procedure, so we do not pay special attention to some constants which, of course,
could grow with respect to $k$ at any formal step. 
For this reason we sometimes avoid the exact computation of some of these constants and we indeed may use the same letter  to denote different ones. 

\subsection{Differential equations for $\f{k},\v_{k}$}\label{FormalEquations}
We are going to describe the equations that $\{\f{k},\v_k\}_{k\geq 0}$ have to satisfy.
As it is usual, the equations for the leading order terms $\f{0}$ and $\v_{0}$ will be nonlinear while the equations
for $\f{k},\v_k$ will be found to be non-homogeneous linear equations.
To shorten the notation we introduce $F(x) = x \lambda(x)$ and $\om(x) = x \omega(x)$ and we denote 
$DF(x)$ and $D\tilde \omega (x)$ to denote the derivatives respect to $x$ of these functions. 

With this notation,  equations \eqref{edo3}, \eqref{edo4} read:
\begin{eqnarray}
0&=& f''+\frac{f'}{r}-f\frac{n^2}{r^2}+F(f)-fv^2,\label{edo33}\\
0&=&fv'+\frac{fv}{r}+2f'v + \qq (f\Omega- \om(f)), \label{edo44}
\end{eqnarray}
and we consider the formal expansions defined in \eqref{exp}:
$$f(r;q) = \sum_{i\geq 0} \f{i}(r) \qq^{2i},\quad v(r;q) = \qq\sum_{i\geq 0} \v_{i}(r)\qq^{2i},\quad
\Omega(q)=\sum_{i\geq 0} \Omega_{i}\qq^{2i}.$$

\begin{proposition}\label{eqfkvk}
The leading order terms $\f{0}$ and $\v_0$, satisfy the equations
\begin{align}
\label{f0}
&0=\f{0}''+\frac{\f{0}'}{r}-n^2\frac{\f{0}}{r^2}+F(\f{0}),\\
\label{v0}
&0= \f{0}\v_0'+\frac{\f{0} \v_0}{r}+2\f{0}'\v_0 +  \f{0}\Omega_0-\tilde \omega(\f{0}).
\end{align}

For $k\geq 1$, $\f{k}$ and $\v_{k}$ satisfy the linear nonhomogeneous equations:
\begin{align}
\label{eq:fk}
&\f{k}''+\frac{\f{k}'}{r}-n^2\frac{\f{k}}{r^2}+DF(\f{0})\f{k}=\b_k(r),\\
\label{eq:vk}
&\f{0}\v_k'+\frac{\f{0} \v_k}{r}+2\f{0}'\v_k +  \f{0}\Omega_k = \bv_k(r),
\end{align}
where
\begin{align}
b_{k}(r)= &-\sum_{i=1}^k D^{2i} F(\f{0}(r)) \sum_{\scriptsize{\begin{array}{c} k_1+\cdots+k_i=k \\ 1\leq k_j\leq k-1\end{array}}} f_{k_1}(r) \cdots f_{k_i}(r) \notag \\
&+\sum_{i=1}^{k} \sum_{l=1}^{i}  f_{k-i}(r) \cdot v_{l-1}(r) \cdot v_{i-l}(r), \label{defbk}
\\ \label{defbvk}
\bv_k(r)=& \sum_{i=0}^{k-1} \left (\f{k-i}(r)\big (\v_i'(r)+ r^{-1}\v_i(r)\big ) + 
2\f{k-i}'(r)\v_i(r) \right ) + \sum_{i=0}^{k-1}  \f{k-i}(r) \cdot \Omega_i \notag
\\
&-\sum_{i=1}^{2k} D^{2i} \om(f_0(r)) \sum_{\scriptsize{\begin{array}{c} k_1+\cdots+k_i=k \\ 1\leq k_j\end{array}}} 
\f{k_1}(r) \cdots \f{k_i}(r)
\end{align}
with
$F(x)=x\lambda (x)$ and $\om(x) =x\omega(x)$.
In particular, $b_k$ is independent of $\f{k}$ and $\v_k$ and $\bv_{k}$ is independent of $\v_k$.
\end{proposition}
\begin{proof}
By substituting expression \eqref{exp} in \eqref{edo33}, one obtains equation \eqref{f0} for $\f{0}$. 
As for $\v$, equation \eqref{edo44}, gives to leading order equation \eqref{v0} for $\v_{0}$.

We now deal with $\f{k},\v_{k}$, $k\geq 1$. 
To illustrate the procedure we start by obtaining the particular equations for $\f{1},\v_1$. 
Expanding equation \eqref{edo33} in powers of $\qq$, the order \BigO{q^{2}} provides an equation for $\f{1}$ in terms of $\v_0$ and $\f{0}$, which reads,
\begin{equation}
\label{eq:f1}
f_1''+\frac{f_1'}{r}-n^2\frac{f_1}{r^2}+DF(f_0)f_1=f_0 v_0^2,
\end{equation}
which gives $b_1(r)=f_0(r)v_0(r)^2$.

Expanding equation \eqref{edo44} in powers of $\qq$, the order \BigO{q^{3}} provides an equation for $\v_1$ in terms of $f_0$, $f_1$ and $v_0$:
\begin{equation}\label{eq:v1}
\f{0}\v_1'+\frac{\f{0} \v_1}{r}+2\f{0}'\v_1 +  \f{0}\big (\Omega_1-\f{1} D\omega(\f{0})\big ) = \bv_1(r)
\end{equation}
with 
\begin{equation}
\label{eq:c1}
\bv_{1}(r) = -\f{1}(r)(\Omega_0-\omega(\f{0}(r))) -\f{1}(r)\v_0'(r)- r^{-1} \f{1}(r) \v_0(r)-2\v_0(r) \f{1}'(r).
\end{equation}

To deal with the general case we first observe that the ansatz \eqref{exp} may also be expressed in terms of a Taylor expansion of $f,v$ and $\Omega$ with respect to $q$. 
Using the expansions \eqref{exp}, we deduce that $\partial_q^{2i+1} f(r;0)=0$, $\partial_{q}^{2i} v(r;0)=0$ and $\partial_{q}\Omega^{2i+1}(0)=0$. 
This yields:
\begin{equation}\label{expTaylor}
\begin{split}
f(r;q) &= \sum_{i\geq 0} \frac{\partial_{q}^{2i}f(r;0)}{(2i)!} \qq^{2i},\quad v(r;q) = \qq\sum_{i\geq 0} \frac{\partial_{q}^{2i+1} v(r;0)}{(2i+1)!} \qq^{2i},\\
\Omega(q)&=\sum_{i\geq 0} \frac{\partial_{q}^{2i} \Omega(0)}{(2i)!}\qq^{2i}.
\end{split}
\end{equation}
Therefore, 
\begin{equation}\label{fkdfk}
f_{k}(r) = \frac{\partial_{q}^{2k}f(r;0)}{(2k)!},\quad v_{k}(r) = \frac{\partial_{q}^{2k+1} v(r;0)}{(2k+1)!}\quad\text{and}\quad \Omega_k=\frac{\partial_{q}^{2k} \Omega(0)}{(2k)!}.
\end{equation}

As a consequence, in order to obtain the equations for $\f{k},\v_k$ and a general expression for $b_k,\bv_k$  
it is enough to differentiate equations \eqref{edo33} and \eqref{edo44} with respect to $q$.
We shall use Faa di Bruno formula along with Leibnitz's rule. 
%
%

We first deal with \eqref{edo33}. 
We must compute the $2k$- derivative with respect to $q$ of the nonlinear term $F(\f{})- \f{} \v^2$
and then evaluate at $q=0$. Using Faa di Bruno's formula and the identity \eqref{fkdfk} gives 
\begin{align*}
\frac{\partial_{q}^{2k} (F\circ f)(r;0)}{(2k)!} = &\sum_{i=1}^{2k} D^i F(f(r;0)) \sum_{\scriptsize{\begin{array}{c} k_1+\cdots+k_i=2k \\ 1\leq k_j\end{array}}} 
\frac{\partial_{q}^{k_1} f(r;0)}{k_1!} \cdots \frac{\partial_{q}^{k_i} f(r;0) }{k_i!}
 \\
= &DF(\f{0}(r))\f{k}(r) + b_k^1 (r),
\end{align*}
where, upon using once more identity \eqref{fkdfk} along with  $\partial_{q}^{2l+1} f(r;0)=0$, $b_k^1$ is found to read:
%
%
\begin{equation}\label{defbk1}
b_k^1(r) := \sum_{i=1}^k D^{2i} F(\f{0}(r))) \sum_{\scriptsize{\begin{array}{c} k_1+\cdots+k_i=k \\ 1\leq k_j\leq k-1\end{array}}} \f{k_1}(r) \cdots \f{k_i}(r).
\end{equation}
We note that the last sum does only depend on $\f{l}$ with $0< l<k$.

We now proceed likewise with $\partial_{q}^{2k} (f\cdot v^2)(r;0)$. Here we also note that $\partial_{q}^{2l}v(r;0)=0$. Then, using Leibnitz rule:
\begin{equation}\label{defbk2}
\begin{split}
b_{k}^2(r) :&= \partial_{q}^{2k} (f\cdot v^2)(r;0) = 
\sum_{i=2}^{2k} \left ( \begin{array}{c} 2k \\ i\end{array}\right ) \partial_q^{2k-i}f(r;0) \cdot \partial_{q}^{i} v^2(r;0) 
\\ &=\sum_{j=1}^{k} \sum_{m=1}^{2j-1} 
\left ( \begin{array}{c} 2k \\ 2j\end{array}\right )\left ( \begin{array}{c} 2j \\ m\end{array}\right )
\partial_{q}^{2k-2j} f(r;0) \cdot \partial_{q}^{m} v(r;0) \cdot \partial_{q}^{2j-m} v(r;0)
\\ &=(2k)!\sum_{j=1}^{k} \sum_{l=1}^{j} 
\f{k-j}(r) \cdot \v_{l-1}(r) \cdot \v_{j-l} (r),
\end{split}
\end{equation}
so $b_{k}^2$ only depends on $\f{l}$, $\v_{l}$ with $0\leq l<k$. 

Using the above expressions for $\partial_q^{2k} (F(f) -f v^2)$, we compute the $2k$-derivative of equation \eqref{edo33} with respect to $q$ and, evaluating this derivative at $q=0$, one finds that $\f{k}(r) = \partial_{q}^{2k}f(r;0)/(2k)!$ is a solution of the linear equation
\begin{equation*}
\f{k}''(r) + \frac{\f{k}'(r)}{r}-\f{k}(r)\frac{n^2}{r^2} + DF(\f{0}(r)) \cdot \f{k}(r) + b_k^1(r) - \frac{b_k^2(r)}{(2k)!}=0.   
\end{equation*}
Therefore, $\f{k}$ satisfies equation \eqref{eq:fk} with $b_k = -b_k^1 + b_k^2 /(2k)!$ having the form \eqref{defbk}.

We now deal with equation \eqref{edo44}. 
The procedure is exactly analogous to the one for equation \eqref{edo33}. 
First, we observe that, using the Leibnitz's rule, as well as identity \eqref{fkdfk}, 
\begin{align*}
\partial^{2k+1}_q & \left ( f(r;0) \cdot v'(r;0) + \frac{f(r;0)\cdot v(r;0)}{r} + 2 f'(r;0)\cdot v(r;0)\right)  \\
 &= (2k+1)! \big[ \f{0}(r) \big (\v_k'(r) + \frac{\v_k(r)}{r}\big ) + 2  \f{0}'(r)\v_k(r)\big ] + \bv^1_k(r),
\end{align*}
where, 
\begin{align*}
\bv_{k}^1(r) = (2k+1)!\sum_{i=0}^{k-1} \left (\f{k-i}(r)\big ( \v_i'(r)+ \frac{\v_i(r)}{r}\big ) + 
2\f{k-i}'(r)\v_i(r) \right ).
\end{align*}

It only remains to compute the $2k+1$-derivative with respect to $q$ of the nonlinear term $qf \big (\Omega(q) - \omega(f)\big )$. 
First, we define $\Ot(q) = q \Omega(q)$ and we compute $\partial_q^{2k+1} \big (f(r;0) \cdot \Ot(0)\big )$. 
We obtain, using Leibnitz's rule, 
\begin{align*} 
\partial_{q}^{2k+1} \big (f(r;0) \cdot \Ot(0) \big ) &=
\sum_{i=0}^{2k+1} \left ( \begin{array}{c} 2k+1 \\ i\end{array}\right ) \partial_q^{2k+1-i}f(r;0) \cdot \partial_{q}^{i} \Ot(0)  
 \\
&= (2k+1)! \f{0} \Omega_k + \bv_k^2(r),
\end{align*}
where
\begin{equation*}
\bv_k^2(r) = (2k+1)!\sum_{i=0}^{k-1} \f{k-i}(r) \cdot \Omega_i .
\end{equation*}

We now introduce $\om(x) = x \omega(x)$ and compute $\bv_k^3(r) := \partial_q^{2k+1} \big( q \om(f(r;q)) \big )_{|q=0}$:
\begin{align*}
\frac{\bv_{k}^3(r)}{(2k+1)!} 
&=\frac{1}{(2k)!} \partial_{q}^{2k} \big (\om(f(r;0))\big ) \\
&= \sum_{i=1}^{2k} D^i \om(f(r;0)) \sum_{\scriptsize{\begin{array}{c} k_1+\cdots+k_i=2k \\ 1\leq k_j\end{array}}} 
\frac{\partial_{q}^{k_1} f(r;0)}{k_1!} \cdots \frac{\partial_{q}^{k_i} f(r;0) }{k_i!}
 \\ &= \sum_{i=1}^{2k} D^i \om(f_0(r)) \sum_{\scriptsize{\begin{array}{c} k_1+\cdots+k_i=k \\ 1\leq k_j\end{array}}} 
\f{k_1}(r) \cdots \f{k_i}(r)
\end{align*}

Finally we compute the $2k+1$-derivative with respect to $q$ of equation \eqref{edo44} and we obtain
that $\v_k(r)= \partial_q^{2k+1} \v(r;0)/(2k+1)!$ satisfies equation \eqref{eq:vk} with $\bv_k$ defined as in
\eqref {defbvk}.
\end{proof}

\subsection{The leading order term}
We have already proved that the leading order terms $\f{0}$ and $\v_0$, along with a suitable choice of $\Omega_0=\omega(1)$ 
have to be solutions of the boundary problem:
\begin{equation}
\label{f0bis}
\begin{split}
&0=\f{0}''+\frac{\f{0}'}{r}-n^2\frac{\f{0}}{r^2}+\f{0}\lambda(\f{0})\\
&\f{0}(0)=0, \quad \lim_{r\to\infty} \f{0}(r) = 1,
\end{split}
\end{equation}
and 
\begin{equation}\label{v0bis}
\begin{split}
&0=\f{0}\v_0'+\frac{\f{0} \v_0}{r}+2\f{0}'\v_0 +  \f{0}(\Omega_0-\omega(\f{0})) \\
&\v_0(0)=0, \quad \lim_{r\to \infty} \v_0(r) <+\infty.
\end{split}
\end{equation}
It is clear that the nonlinear equation for $\f{0}$ is qualitatively different to the ones for $\f{k}$ with $k\geq 1$, which are all of them
nonhomogeneous linear equations. 
Moreover, in order to begin an induction procedure (which will be our strategy to prove Theorem \ref{main}) we 
also need to prove the existence and properties of $\v_{0}$. 
For that reason we study the leading order terms separately. 
Next proposition  proves the part of Theorem \ref{main} related to $\f{0}$ and $\v_0$.
\begin{proposition}\label{aefderf}
The boundary problem \eqref{f0bis} has a bounded solution $\f{0}>0$. 
Moreover, $\f{0}$ satisfies the following inequalities
\begin{equation*}
0<r\f{0}'(r) \leq  n^2 \f{0}(r),\qquad r>0,
\end{equation*}
and it has the asymptotic expansions,
\begin{equation}
\label{cotaf0}
\f{0}(r) = \alpha r^{n} + \BigO{r^{n+1}}, \;\; \text{as}\;\;r\to 0 \;\; \mbox{and}  \;\;\f{0}(r)  = 1 -\frac{n^2}{d r^2} + \BigO{r^{-4}},\;\;  \text{as}\;\;r\to +\infty.
\end{equation}
with $d= -\lambda'(1)$. 
We also have that $\f{0}'\in \BiO{3}{n-1}$ with $\lim_{r\to +\infty} r^3 \f{0}'(r) = 2\frac{n^2}{ d}$ and $\f{0}''\in \BiO{4}{\min\{0,n-2\}}$.

The problem \eqref{v0bis} has a bounded solution $\v_0$ if and only if $\Omega_0=\omega(1)$. 
Moreover, 
\begin{equation}
\label{v0def}
\v_0(r) = (r\f{0}^2(r))^{-1}\int_0^r t\f{0}(t)^2(\omega(\f{0}(t))-\Omega_0)\,dt, 
\end{equation}
and it satisfies the asymptotic expansions
\begin{equation}
\label{v0inf}
\v_0(r) = C r + \BigO{r^2}\;\;  \text{as}\;\;r\to 0,\;\;
\v_0(r) = -\frac{n^2\omega'(1)\log(r)}{d r} + \BigO{1/r} \;\;  \text{as}\;\;r\to +\infty.
\end{equation}
Moreover, $\v_0'\in \BiOv{2}{0}{1}$, $\v_0''\in \BiOv{3}{0}{1}$.

When $\omega$ is a monotone function the solution $v_0$ has constant sign.
\end{proposition}
\begin{proof}
As it is shown in \cite{Aguareles2011}, the boundary problem  \eqref{f0bis} has a unique bounded solution.
The inequalities and the expansions for $\f{0}(r)$ were rigourosly proven in \cite{Aguareles2011} 
for the case that $\lambda$ is an analytic function (Section 2 as $\r \to 0$ and Sections 4 and 5 as $r\to +\infty$). 
The expansion for $r$ small enough is also true in the case $\lambda \in \co^{\infty}$  and the behaviour as $r\to +\infty$ 
can be straightforwardly deduced from Lemma 2.14 and Remark 2.15 in \cite{Aguareles2011}. 
Then, re-writting equation  \eqref{f0bis} we get the identity 
$(r \f{0}'(r))' = n^2 r^{-1}\f{0}(r) - \f{0}(r)\lambda(\f{0}(r))$. 
From this identity we deduce the asymptotic expansions for $\f{0}'$ and 
$\f{0}''$.

We also know that $\lim_{r\to +\infty} r^{3} \f{0}'(r)$ exists from the previous work in \cite{Aguareles2011}. 
To compute this limit, we use L'H\^opital's rule, and the asymptotic expansion for $\f{0}$:
\begin{equation*}
\frac{n^2}{d} = \lim_{r\to +\infty} r^2 \big (1-\f{0}(r)\big ) =  \lim_{r\to +\infty} r^2 \int_{r}^{+\infty} f'(\xi)\, d\xi =
 \lim_{r\to +\infty} \frac{r^3 f'(r)}{2}
\end{equation*}
and so the results for $\f{0}$ are proven.

As for $\v_0$, since $\v_0(0)=0$ and it satisfies equation \eqref{v0bis}, using property \eqref{propv}, gives to leading order expression \eqref{v0def}. 
Now, using the asymptotic behaviour of $\f{0}(r)$ as $r\to\infty$ in \eqref{v0def}, gives 
\begin{equation*}
\begin{split}
\v_0(r)  = &\big(r-2 n^2/(d^2r)+o(r^{-1})\big)^{-1}\left(\int_0^{r_0} t\f{0}(t)^2(\omega(\f{0}(t))-\Omega_0)\,dt\right.\\
&+ \left.\int_{r_0}^r t\left(\omega(1)-\Omega_0-\frac{n^2}{t^2\,d}\big(2(\omega(1)-\Omega_0)+\omega'(1)\big)+\LitO{t^{-2}} \right) \, dt \right),
\end{split}
\end{equation*}
provided $r\geq r_0$ and $r_0$ is large enough. This last expression shows that in order for $v_0$ to be bounded at infinity, 
we have to impose $\omega(1)=\Omega_0$ and so this gives the asymptotic behaviour of $\v_0(r)$ as $r\to\infty$ presented in \eqref{v0inf}.


Also, the asymptotic behaviour of $\v_0(r)$ as $r\to 0$ is easily obtained by using the asymptotic expression of $\f{0}$ in equation \eqref{v0def},
\begin{equation*}
\v_0(r)\sim\frac{\int_0^r t\alpha^2 t^{2n}(\omega(\alpha t^n)-\omega(1))\,dt}{\alpha^2 r^{2n+1}}=\frac{\omega(0)-\omega(1)}{2n+2} r + \BigO{r^2}.
\end{equation*}


The asymptotic behaviour of both $\v_0'$ and $\v_0''$ follows from the fact that $\v_0\in \BiOv{1}{1}{1}$ is a solution of equation \eqref{v0bis}
along with the asymptotic behaviour of $\f{0},\f{0}'$.

It only remains to check that $\v_0(r)$ has constant sign when $\omega(x)$ is a monotone function. 
For instance, according to \eqref{v0def} if $\omega(x)$ is decreasing, since $\Omega_0=\omega(1)$, $\omega(0)-\Omega_0\geq \omega(f_0(t))-\Omega_0\geq \omega(1)-\Omega_0=0$, and hence $\v_0(r)\geq 0$ for all $r\geq 0$. Likewise, if $\omega(x)$ is increasing, $\v_0(r)\leq 0$ for all $\r\geq 0$.
\end{proof}
\subsection{Existence and properties of $\f{k}$. An induction procedure}
In this section we are going to prove the results of Theorem \ref{main}
related to $\f{k},\v_k$ for $k\geq 1$. We will use the notation and results from
Proposition \ref{eqfkvk}. More precisely, we will prove that the problems:
\begin{equation}
\label{eq:fkbis}
\begin{split}
&\f{k}''+\frac{\f{k}'}{r}-n^2\frac{\f{k}}{r^2}+DF(\f{0}(r))\f{k}=\b_k(r)\\
&\f{k}(0)=0, \qquad  \f{k}(r) \;\;\text{bounded}\;\; r\geq0
\end{split}
\end{equation}
with $F(x) = x \lambda(x)$, and 
\begin{equation}
\label{eq:vkbis}
\begin{split}
&\f{0}\v_k'+\frac{\f{0} \v_k}{r}+2\f{0}'\v_k +  \f{0}\Omega_k = \bv_k(r), \\
&\v_k(0)=0,\qquad v_{k}(r) \;\;\text{bounded}\;\; r\geq0
\end{split}
\end{equation}
have solutions $\f{k}$ and  $
\v_k$ provided 
\begin{equation}
\Omega_k = 0.
\end{equation}
Recall that $\b_k$,  $\bv_k$, were defined in Proposition \ref{eqfkvk}, for  $k\ge 1$.
To prove this result we will use 
an induction procedure. 

We first recall that, if $\f{0}, \f{1}\cdots, \f{k-1}$ and $\v_{0},\v_{1}, \cdots, \v_{k-1}$ are known, then, the independent term
$\b_{k}$ of \eqref{eq:fkbis} is determined and henceforth $\f{k}$ satisfies a linear non-homogeneous equation.
If we are able to prove the existence of such a solution, then, by property \eqref{propv} and taking into account that
$\v_k(0)=0$, we will have an explicit expression for $\v_k$ which depends on $\Omega_k$ and $\bv_k$:
\begin{equation}\label{formulavk}
\v_k(r) = \big (r\f{0}^2(r)\big )^{-1}\int_{0}^r t \f{0}(t) \big (\bv_k(t) - \f{0}(t)\Omega_k\big)\, dt
\end{equation}
Recall here that $\bv_k$ depends only on $\f{0},\cdots, \f{k}$ and $\v_{0},\cdots, \v_{k-1}$.

Therefore, once one knows how to solve the equation for $\f{k}$, the function $\v_{k}$ is totally determined.
Since all the equations for $\f{k}$ have the same shape, it is mandatory to study the existence of solutions of linear equations of the form
\begin{equation}\label{eqfkgeneral}
\begin{split}
&g''(r)+\frac{g'(r)}{r}-n^2\frac{g(r)}{r^2}+DF(\f{0}(r))g(r)=h(r)\\
&g(0)=0,\qquad g(r) \;\;\text{bounded}\;\; r\geq0.
\end{split}
\end{equation}
We state the following technical lemma which will be proven in Subsection \ref{FixedPointEquation} by using the Fixed Point Theorem in a suitable
Banach space.

\begin{lemma}\label{soleqfkgeneral}
Let $h:[0,+\infty)\to \R$ be a $C^2$ function. We define
\begin{equation*}
\mathcal{E}[h](r):=h''(r) + \frac{h'(r)}{r} - h(r)\frac{n^2}{r^2} + \big[DF(\f{0}(r))+d\big] h(r)
\end{equation*}
with $F(x)=x\lambda(x)$ and $d= -\lambda'(1)$. Assume that $\mathcal{E}[h] \in \BiO{3}{n-1}$, that is:
\begin{equation}\label{propheqgeneral}
\mathcal{E}[h](r) =\BigO{r^{n-1}},\;\; r \to 0,\qquad \mathcal{E}[h](r) (r) = \BigO{r^{-3}},\;\; r\to +\infty.
\end{equation}
Then there exists a unique bounded solution $g$ of the boundary problem \eqref{eqfkgeneral}. Moreover,
if $\Delta g= g+hd^{-1}$, we have that
\begin{equation*}
\Delta g\in \BiO{3}{n},\;\; \Delta g'\in \BiO{3}{n-1},\;\; \text{and}\;\; \Delta g''\in \BiO{3}{\max\{n-2,0\}}.
\end{equation*}
In particular, $\lim_{r\to +\infty} g(r) = 0$. 
\end{lemma}

Now we begin our induction scheme. 
We begin with $\f{1}$ which satisfies the equation \eqref{eq:f1}, that is:
\begin{equation*}
\f{1}''+\frac{\f{1}'}{r}-n^2\frac{\f{1}}{r^2}+DF(\f{0})\f{1}=\b_1(r)
\end{equation*}
with $\b_1(r) = \f{0}(r) \v_{0}^2(r)$.
We want to apply Lemma \ref{soleqfkgeneral} and
for that we check that $\mathcal{E}[\b_1] \in \BiO{3}{n-1}$. 
We point out that, by Proposition \ref{aefderf},
$\b_1= \f{0} \v^{2}_0 \in \BiOv{2}{n+2}{2}$, $\b_1' \in \BiOv{3}{n+1}{2}$, $\b_1''\in \BiOv{4}{n}{2}$
and, consequently:
\begin{equation*}
\b_1''(r) + \frac{\b_1'(r)}{r} - \b_1(r) \frac{n^2}{r^2} \in \BiOv{4}{n}{2}.
\end{equation*}

In addition, $\big [DF(\f{0}(r))+d\big] \b_1(r) = \BigO{r^{n+2}}$ as $\r\to 0$ and, since $DF(1)=\lambda'(1)=-d$, and using Proposition \ref{aefderf}
 for the asymptotics of $\f{0}$ as $r\to \infty$
 \begin{equation}\label{expDF(f0)}
\big [DF(\f{0}(r))+d\big ]= \BigO{\f{0}(r)-1} = \BigO{r^{-2}}
\end{equation}
and this gives $\big [DF(\f{0}(r))+d\big ] \b_{1}(r)= \BigO{r^{-4}\log^2 r}$.
Therefore we conclude that
$\mathcal{E}[\b_1]\in \BiOv{4}{n}{2} \subset \BiO{3}{n-1}$. Then, Lemma \ref{soleqfkgeneral} gives the existence of a solution $\f{1}$ of problem \eqref{eq:fkbis} for $k=1$ with
$\Delta \f{1} = \f{1}+d^{-1} \b_{1}$ 
satisfying
\begin{equation*}
\Delta \f{1} \in \BiO{3}{n}\;\; \Delta \f{1}'\in \BiO{3}{n-1},\;\; \text{and}\;\; \Delta \f{1}''\in \BiO{3}{\max\{n-2,0\}},
\end{equation*}
which gives:
\begin{equation*}
\f{1}\in \BiOv{2}{n}{2}\;\; \f{1}' \in \BiOv{3}{n-1}{2},\;\; \text{and}\;\; \f{1}''\in \BiO{3}{\max\{n-2,0\}}.
\end{equation*}

Now we deal with $\v_1$ and $\Omega_1$. As we state in \eqref{formulavk},
$$
\v_1(r) = \big (r\f{0}^2(r)\big )^{-1}\int_{0}^r t \f{0}(t) \big (\bv_1(t) - \f{0}(t)\Omega_1\big)\, dt,
$$
with $\bv_1$ defined in Proposition \ref{eqfkvk}, formula \eqref{eq:c1}. 
Using that $\f{1}\in \BiOv{2}{n}{2}$, $\f{1}' \in \BiOv{3}{n-1}{2}$,
along with $\v_0 \in \BiOv{1}{1}{1}$ and $\v_0' \in \BiOv{2}{0}{1}$,
we have that $\bv_1 \in \BiOv{2}{n}{2}$. 
Therefore, $\v_1$ will be a bounded solution if and only if
$r \f{0}(r) \big (\bv_1(r) - \f{0}(r)\Omega_1\big)$ is a bounded function. 
This implies that
$$
0= \lim_{r\to +\infty} \bv_1(r) - \f{0}(r)\Omega_1 = \Omega_1.
$$
Hence we actually have that
$$
\v_1(r) = \big (r\f{0}^2(r)\big )^{-1}\int_{0}^r t \f{0}(t) \bv_1(t)\, dt.
$$
Now we need to compute the asymptotic behaviour of $\v_1$. 
Clearly, for $r\to 0$, since $\bv_1 \in \BiOv{2}{n}{2}$,
$\v_1(r) = \BigO{r}$. Now we deal with $r\to +\infty$. We notice that, if $r_0$ is big enough,
$$
\left |\int_{r_0}^r t \f{0}(t) \bv_1(t)\, dt \right | \leq C \int_{r_0}^r  \frac{\log^2 t}{t} \, dt \leq C \log^3 r.
$$
Then
$$
\vert \v_1(r) \vert \leq C \frac{\log^3 r}{r},\qquad \text{as}\;\; r\to +\infty.
$$
Summarizing, $\v_1 \in \BiOv{1}{1}{3}$. 
Moreover, from \eqref{eq:vkbis} with $k=1$:
$$
\f{0}(r)\v_{1}'(r) = \bv_1(r) -2 \f{0}'(r) \v_1(r) - r^{-1} \f{0}(r) \v_1(r)
$$
which implies that $\v_1' \in \BiOv{2}{0}{3}$. We can also deduce that $\v_1''\in \BiOv{3}{0}{3}$.

Now we state the induction hypothesis: the unique bounded solution $\f{k-1}$, $k\geq 2$, of problem \eqref{eq:fkbis} satisfies
\begin{equation}\label{inductionf}
\f{k-1} \in \BiOv{2}{n}{2(k-1)},\;\; \f{k-1}' \in \BiOv{3}{n-1}{2(k-1)},\;\;\f{k-1}''\in \BiO{3}{\max\{n-2,0\}}.
\end{equation}
Moreover, problem \eqref{eq:vkbis} has bounded solution $\v_{k-1}$ if and only if $\Omega_{k-1}=0$ and in this case,
\begin{equation}\label{inductionv}
\v_{k-1} \in \BiOv{1}{1}{2k-1},\;\; v_{k-1}'\in \BiOv{2}{0}{2k-1},\;\; v_{k-1}''\in \BiOv{3}{0}{2k-1}.
\end{equation}

We begin first by checking that $\b_{k} \in \BiOv{2}{n+1}{2k}$ and $ \bv_{k} \in \BiOv{2}{n}{2k}$.
Indeed, by induction hypothesis \eqref{inductionf} and \eqref{inductionv} and formula \eqref{defbk} for $b_k$, we have that
\begin{equation*}
\b_{k} \in \BiOv{2}{2n}{2k} \cap \BiOv{2}{n+2}{2k} \subset  \BiOv{2}{n+1}{2k}.
\end{equation*}
We emphasize that if $n=1$, $2n<n+2$, but if $n\geq 2$, $2n\geq n+2$. 
To unify both cases we have considered $\b_{k} \in \BiOv{2}{n+1}{2k}$.
Analogously one see that $\bv_{k}\in \BiOv{2}{n}{2k}$.

In order to compute the orders for $\b_{k}'$ and $\bv_{k}'$ we take into account that, by induction hypothesis if $l\leq k-1$, the functions 
$\f{l}'(r),\v_{l}'(r),\v_{l}''(r)$ are of order of $\f{l}(r)r^{-1}$, $\v_{l}(r) r^{-1}$ and $\v_{l}(r)r^{-2}$ respectively, so the same happens for
the products of these functions. 
Moreover, $\f{l}'' \in \BiO{3}{\max\{n-2,0\}}$, for $l\leq k-1$. 
Then, tedious but easy computations yield:
\begin{equation}\label{orderbk1}
\b_{k} \in \BiOv{2}{n+1}{2k},\;\;\b_{k}'\in \BiOv{3}{n}{2k},\;\;\b_{k}'' \in \BiO{3}{n-1}
\end{equation}
and
\begin{equation}\label{orderck1}
\bv_{k}\in \BiOv{2}{n}{2k},\;\;\bv_{k}'\in \BiOv{3}{n-1}{2k}.
\end{equation}

The first consequence is that $\mathcal{E}[\b_k] \in \BiO{3}{n-1}$ and hence by Lemma \ref{soleqfkgeneral} there exists a unique solution $\f{k}$
of problem \eqref{eq:fkbis} satisfying that
\begin{equation*}
\f{k}+d^{-1}\b_{k} \in \BiO{3}{n},\;\; \f{k}'+d^{-1}\b_{k}' \in \BiO{3}{n-1},\;\;\f{k}''+d^{-1} \b_{k}''\in \BiO{3}{\max\{n-2,0\}}
\end{equation*}
and taking into account the expansions of $\b_{k}$ in \eqref{orderbk1}, the induction hypothesis \eqref{inductionf} is fullfilled for $\f{k}$.

Finally we deal with $\v_{k}$. 
We proceed likewise as $\v_{1}$. 
From identity \eqref{formulavk},
$$
\v_{k}(r) = \big (r\f{0}^2(r)\big )^{-1}\int_{0}^r t \f{0}(t) \big (\bv_{k}(t) - \f{0}(t)\Omega_{k}\big)\, dt,
$$
$\v_{k}$ will be a bounded solution if and only if
$r \f{0}(r) \big (\bv_{k}(r) - \f{0}(r)\Omega_{k}\big)$ is a bounded function and consequently, 
since $\bv_{k} \in \BiOv{2}{n}{2k}$, 
$$
0= \lim_{r\to +\infty} \bv_{k}(r) - \f{0}(r)\Omega_{k} = \Omega_{k}.
$$
Therefore the induction hypothesis for $\Omega_{k}$ is also satisfied. We rewrite $\v_{k}$ as
$$
\v_{k}(r) = \big (r\f{0}^2(r)\big )^{-1}\int_{0}^r t \f{0}(t) \bv_{k}(t)\, dt,
$$
and compute the asymptotic behaviour of $\v_{k}$. 
Since $\bv_{k}(r),\f{0}(r) = \BigO{r^{n}}$ as $r\to 0$, one deduces that
$\v_{k}(r) = \BigO{r}$. 
As in the case $k=1$, if $r_0$ is big enough,
$$
\left |\int_{r_0}^r t \f{0}(t) \bv_{k}(t)\, dt \right | \leq C \int_{r_0}^r  \frac{\log^{2k} t}{t} \, dt \leq C \big (\log r\big )^{2k+1}
$$
and hence
$$
\vert \v_{k}(r) \vert \leq C \frac{\log^{2k+1} r}{r},\qquad \text{as}\;\; r\to +\infty.
$$
Summarizing, $\v_{k} \in \BiOv{1}{1}{2k+1}$. 
Moreover,
$$
\f{0}(r)\v_{k}'(r) = \bv_{k}(r) -2 \f{0}'(r) \v_{k}(r) - r^{-1} \f{0}(r) \v_{k}(r)
$$
implies that $\v_{k}' \in \BiOv{2}{0}{2k+1}$ and we finally deduce that $\v_{k}''\in \BiOv{3}{0}{2k+1}$ by using
\eqref{orderck1}.

This ends the proof of Theorem \ref{main}. 

\subsection{Proof of Lemma \ref{soleqfkgeneral}}\label{FixedPointEquation}
We first write equation \eqref{eqfkgeneral} in a more suitable way, i.e. as a fixed point equation.
Adding and subtracting the term
$d\, g$, where $d=-\lambda'(1)$, which is positive since $\lambda'(1)<0$, performing the change of variables $s=\sqrt{d}\,r$ and denoting by $\ga(s)=g(s/\sqrt{d})$, $\ha(s)= d^{-1} h(s/\sqrt{d})$ yields
\begin{equation}
\label{bessel}
\begin{split}
&\ga''(s)+\frac{\ga'(s)}{s}-\ga(s)\left(\frac{n^2}{s^2}+1\right)= \ha(s) -\ga(s)\left[\frac{DF(\ft_0(s))}{d}+1\right ]\\
&\ga(0)=0,\qquad \ga(s)\;\; \text{bounded} \;\;s\geq 0,
\end{split}
\end{equation}
where we call $\ft_0(s) =f_0(s/\sqrt{d})$.

As we showed in \eqref{expDF(f0)}
\begin{equation}\label{limsinfDelta}
d^{-1}DF(\ft_0(s))+1= \BigO{s^{-2}},\;\;\; \text{as}\;\;\; s \to +\infty.
\end{equation}
This implies that the dominant term of equation \eqref{bessel} as $s\to \infty$ is the singular equation $-\tilde g(s)= \tilde h(s)$, therefore, it is natural to write
$\ga= -\ha+\Delta g $ with $\Delta g$ being a solution of
\begin{equation}\label{eqDeltag}
\Delta g''+\frac{\Delta g'}{s}-\Delta g\left(\frac{n^2}{s^2}+1\right)=-\tilde{\mathcal{E}}[\tilde h](s)-\Delta g\left(\frac{DF(\ft_{0}(s))}{d}+1\right),\\
\end{equation}
where $\tilde {\mathcal {E}}[\tilde h]$ defined by
$$
\tilde {\mathcal {E}}[\tilde h](s)=\frac{1}{d^2}\mathcal {E}[h](s/\sqrt{d})=\tilde h''(s)+\frac{\tilde h'(s)}{s}
-\tilde h(s)\frac{n^2}{s^2}+\left[\frac{DF(\ft_0(s))}{d}+1\right]\tilde h(s).
$$
Recall that the operator $\mathcal{E}$ is defined in the statement of the lemma.
The boundary conditions are $\Delta g(0)=0$ and $\Delta g(s) $ is  bounded for $s\geq 0$.

Our goal now is to write equation \eqref{eqDeltag} as a fixed point equation.
We emphasize that the dominant part of this equation is the left hand side. Indeed, on the one hand, using \eqref{limsinfDelta},
one sees that the linear term in the right hand side of equation \eqref{eqDeltag}, contributes a small quantity to the equation for large values of $s$,  being the left hand side of equation \eqref{eqDeltag} the dominant part as $s \to \infty$.
On the other hand, as $s\to 0$, even if this linear term is of order one, the dominant part of equation \eqref{bessel} is provided by the first three terms of the left hand side, that is $\Delta g''(s)+\Delta g'(s)/s-\Delta g(s) n^2/s^2$, and so the right hand side in \eqref{bessel} is also relatively small for small values of $s$.

To obtain a fixed point equation we note that the homogeneous modified Bessel equation:

$$
\varphi''(s)+\frac{\varphi'(s)}{s}-\varphi(s)\left(\frac{n^2}{s^2}+1\right)= 0
$$
has two well-known linearly independent solutions, namely $I_n(s)$ and $K_n(s)$ known as the modified Bessel functions of the first and second kind respectively (see \cite{abramowitz}).
Hence, a fundamental matrix of solutions of the homogeneous equation corresponding to equating to zero the left hand side in \eqref{bessel} reads,
\[
 M=\left(
 \begin{array}{cc}
 K_{n}(s) & I_{n}(s)\\
 K'_{n}(s) & I'_{n}(s)
 \end{array}\right),
\]
whose Wronskian is known to be $W(K_{n}(s),I_{n}(s))=1/s$.
We denote by
\begin{equation}
\label{Rk}
\begin{split}
\mathcal{R}[\Delta g](s)=&\tilde {\mathcal {E}}[\tilde h](s)+\Delta g(s)\left(\frac{DF(\ft_0(s))}{d}+1\right).
\end{split}
\end{equation}
We recall here that $\Delta g$ has to be a bounded solution of problem \eqref{eqDeltag} with boundary condition $\Delta g(0)=0$.
Therefore, using the variation of parameters formula, equation \eqref{eqDeltag} becomes a fixed point equation:
\begin{equation}
\label{F1}
\Delta g(s) = \opf1[\Delta g](s):= K_{n}(s)\int_{0}^{s}\xi I_{n}(\xi) \mathcal{R}[\Delta g](\xi)\, d\xi\\
+I_{n}(s)\int_{s}^{\infty}\xi K_{n}(\xi) \mathcal{R}[\Delta g](\xi)\, d\xi.
\end{equation}

In order to prove the existence of the solution of \eqref{F1} (and consequently of problem \eqref{bessel}), we will prove that the linear operator
$\mathcal{F}$ is contractive in some appropriate Banach space $\mathcal{X}$.
However to guarantee the uniqueness of this solution in the space of bounded functions, we need
to carefully study the following linear operator:
\begin{equation}\label{def:T}
\mathcal{T}[\psi](s) = K_{n}(s)\int_{0}^{s}\xi I_{n}(\xi) \psi(\xi)\, d\xi
+I_{n}(s)\int_{s}^{\infty}\xi K_{n}(\xi)\psi(\xi)\, d\xi,
\end{equation}
where $\psi$ is a function defined on $J=[0,+\infty)$.
We notice that $\opf1 = \mathcal{T}\circ \mathcal{R}$.

The operators $\mathcal{T}$ and $\opf1$ are studied in the lemmas \ref{lemma:Top} and \ref{prop:f1} whose  proofs are  deferred to the end of this section.

\begin{lemma}\label{lemma:Top}
Let $\mathcal{T}$ be the linear operator defined in \eqref{def:T}.
Let $\psi$ be a function defined on $J=[0,+\infty)$.
We take $0\leq m<n-1$ and $l\geq 0$.
Then
\begin{equation*}
\psi = \BiO{l}{m} \Longrightarrow \mathcal{T}[\psi] = \BiO{l}{m+2},
\end{equation*}
where the notation $\BiO{l}{m}$ was introduced in \eqref{notation}.
In particular, if $\psi$ is bounded, then $\mathcal{T}(\psi)\in \BiO{0}{2}$.

In the cases $\psi = \BiO{l}{n-1}$ or $\psi = \BiO{l}{n}$ we can only conclude that  $\mathcal{T}[\psi] = \BiO{l}{n}$.

In addition, if $\psi \in \co^{i}(J)$, then $\mathcal{T}[\psi] \in \co^{i+1}(J)$ and
\begin{equation*}
\psi = \BiO{l}{m} \Longrightarrow \mathcal{T}[\psi]'= \BiO{l}{m+1}.
\end{equation*}
In the cases  $\psi = \BiO{l}{n-1}$ or $\psi = \BiO{l}{n}$ we conclude  $\mathcal{T}[\psi]' = \BiO{l}{n-1}$.
\end{lemma}

We now define the Banach space where the solution $\Delta g$ will belong.
We consider  the weight function
\begin{equation}\label{defweigthfunction}
\w(s) = f'_0\big(s/\sqrt{d}\big )
\end{equation}
and the functional space
\begin{equation}
\label{banach}
\mathcal{X}=\{\varphi :J\to\mathbb{R}, \quad \varphi \in \co^0(J),\quad \left|\frac{\varphi (s)}{\w(s)} \right|<+\infty\}.
\end{equation}
We endow $\mathcal{X}$ with the norm
\begin{equation*}
\Vert \varphi \Vert_{\w} =\sup_{s\geq 0} \left | \frac{\varphi(s)}{\w(s)}\right |,
\end{equation*}
and it becomes a Banach space. In addition, since by Proposition \ref{aefderf}, $\w\in \BiO{3}{n-1}$, 
\begin{equation}\label{BanachX}
\mathcal{X} = \BiO{3}{n-1} \cap \,\co^{0}(J).
\end{equation}
\begin{lemma}
\label{prop:f1}
For any given $ \varphi \in \mathcal{X}$, let $\opf1_{\varphi}$ be the linear operator defined by \eqref{F1}:
\begin{equation}
\label{opF1}
\opf1_{\varphi}[g](s)= K_{n}(s)\int_{0}^{s}\xi I_{n}(\xi) \mathcal{R}_\varphi[g](\xi)\, d\xi\\
+I_{n}(s)\int_{s}^{\infty}\xi K_{n}(\xi) \mathcal{R}_\varphi[g](\xi)\, d\xi
\end{equation}
where, analogously to \eqref{Rk}, we denote  $\mathcal{R}_\varphi$  by
\begin{equation*}
\mathcal{R}_\varphi[g](s)= \varphi (s)+g(s)\left(\frac{DF(\ft_0(s)}{d}+1\right).
\end{equation*}

Then,
\begin{itemize}
\item[(i)] If $g\in \mathcal{X}$, then $\opf1_{\varphi}[g] \in \co^1(J)$ and  $\opf1_{\varphi}[g],\opf1_{\varphi}[g]'\in \mathcal{X}$.
In fact $\opf1_{\varphi}[g] (s) \in  \BiO{3}{n}$.
\item [(ii)] $\opf1_{\varphi}$ is contractive in $\mathcal{X}$.
\end{itemize}
\end{lemma}

\begin{proof}[End of the proof of Lemma \ref{soleqfkgeneral}]
We have to deal with both, existence and uniqueness of solutions of problem \eqref{bessel}.
We recall that we look for $\ga$ as $\ga= -\ha+\Delta g $, with $\Delta g$ being a solution of the fixed point equation
$\Delta g= \opf1[\Delta g] = \mathcal{T}[ \mathcal{R}[\Delta g]]$ given in \eqref{F1}.
For the existence we will use mainly Lemma \ref{prop:f1} where $\varphi(s)= \tilde{\mathcal {E}}[\tilde h](s)$.
Then, hypothesis \eqref{propheqgeneral} of Lemma \ref{soleqfkgeneral} and the fact that $\ft_{0}' \in \BiO{3}{n-1}$,
assure that $\tilde{\mathcal{E}}[\tilde h]$ belongs to $\mathcal{X}$ and henceforth
Lemma \ref{prop:f1} provides us with a solution
$\Delta g\in \mathcal{X}$ such that $\Delta g \in \BiO{3}{n}$, $\Delta g' \in \BiO{3}{n-1}$. In addition, since $\Delta g$ is a solution of
the differential equation \eqref{eqDeltag}, $\Delta g'' \in \BiO{3}{\max\{n-2,0\}}$.

Now it only remains to check that $\ga= -\ha+ \Delta g$ is the unique bounded solution of our problem or
equivalently, we see that $\Delta g$ is the only bounded solution of \eqref{eqDeltag}.
Let ${\Delta \bar g}$ be a bounded solution of equation \eqref{eqDeltag}. Then it has to be solution of the
fixed point equation ${\Delta \bar g} = \opf1[{\Delta \bar g}] = (\mathcal{T} \circ \mathcal{R})[{\Delta \bar g}]$.
 We note that
 $$
 \mathcal{R}[{\Delta \bar g}](s) \leq C |{\Delta \bar g}(s)| + |\tilde {\mathcal{E}}[\tilde h](s)|.
 $$
Therefore, since at least ${\Delta \bar g}$ is bounded and $\tilde{\mathcal{E}}[\tilde h]\in \mathcal{X}$,
Lemma \ref{lemma:Top} with $l=m=0$ implies that
${\Delta \bar g}(s) = \BigO{s^2}$ as $s\to 0$, then applying iteratively this lemma,
we obtain that ${\Delta \bar g}(s)=\BigO{s^n}$
as $ s\to 0$.
In particular, since $\w(s)\in\BiO{3}{n-1}$:
\begin{equation}\label{uniques0}
|{\Delta  \bar g}(s) |\leq C\w(s) ,\qquad \text{as}\qquad s\to 0.
\end{equation}

Now we study the behaviour of ${\Delta \bar g}$ as $s\to +\infty$.
We first recall that, according to \eqref{limsinfDelta}
$
d^{-1}DF(\ft_{0}(s)) +1 = O(s^{-2}).
$
Then, since $\tilde {\mathcal{E}}[\tilde h]\in \mathcal{X}\subset \BiO{3}{n-1}$ and ${\Delta \bar g}(s)\in\BiO{0}{n} $,
we conclude that $\mathcal{R}[{\Delta \bar g}]\in\BiO{2}{n-1}$.
Now we apply Lemma \ref{lemma:Top} and we obtain
that ${\Delta \bar g}= \opf1[\bar{\Delta g}] = \mathcal{T}\big [ \mathcal{R}[{\Delta \bar g}]\big ]  \in \BiO{2}{n}$.
Therefore, repeating the previous argumentation, since $\tilde {\mathcal{E}}[\tilde h]\in \BiO{3}{n-1}$ and that $\Delta {\bar g} \in \BiO{2}{n}$  we get that $\Delta {\bar g} \in \BiO{3}{n}$.
Hence, as ${\Delta \bar g}\in \mathcal{X}$ and  $\mathcal{F}$ is a contractive operator over $\mathcal{X}$,
${\Delta \bar g}=\Delta g$.
\end{proof}

The remaining part of this Section is devoted to prove the technical lemmas \ref{lemma:Top} and \ref{prop:f1}.

\subsubsection{Proof of Lemma \ref{lemma:Top}}

Let $m<n-1$, $l\geq 0$ and $\psi:J=[0,+\infty) \to \mathbb{R}$ be a  function in $ \BiO{l}{m}$.
To study the behavior of $\mathcal {T}[\psi]$ (see \eqref{def:T}) as $s\to 0$ and $s\to \infty$,  we recall the asymptotic expansions of
the modified Bessel functions $K_n, I_n$ and their derivatives.

When $s\to 0$ one has:
$$
K_{n}(s) \sim \frac{\Gamma(n)}{2} (s/2)^{-n}\quad  I_{n}(s)\sim \frac{1}{\Gamma(n+1)} (s/2)^{n},
$$
$$
K'_{n}(s) \sim -\frac{n\Gamma(n)}{4} (s/2)^{-n-1},\quad I'_{n}(s)\sim \frac{n}{2\Gamma(n+1)}(s/2)^{n-1}  .
$$
And when $s\to \infty$:
$$
K_{n}(s) \sim e^{-s}\sqrt{\pi/2s}, \quad I_{n}(s) \sim e^s/\sqrt{2\pi s}
$$
$$
K'_{n}(s) \sim -e^{-s}\sqrt{\pi/2s},\quad I'_{n}(s) \sim e^s/\sqrt{2\pi s}.
$$
From now on we will use the expansions of the Bessel functions without explicit mention.

We start by proving the behaviour of $\mathcal{T}[\psi]$
as $s\to 0$.
As  $\psi \in \BiO{l}{m}$, there exists $C>0$ such that $\vert \psi(s) \vert \leq C s^{m}$ for any $s\in J$.
Let $s_0>0$ be such that the above expansions for $s \to 0$ are true  for $0\leq s<s_0$.
We have that
\begin{align*}
\vert \mathcal{T}[\psi](s)\vert \leq & C  K_n(s) \int_{0}^s \xi I_n(\xi) \xi^m \,d \xi + C I_n(s) \int_{s}^{+\infty} \xi K_n(\xi)\xi^m\,d \xi \\
\leq & c\left( s^{-n} \int_{0}^s \xi^{n+1+m} \,d\xi +  s^{n} \int_{s}^{s_0} \xi^{-n+1+m} \, d\xi \right.
+\left. s^{n} \int_{s_0}^{+\infty} \xi ^{1+m}K_n(\xi)\, d\xi \right)
\\ \leq & \bar{c} \max\{ s^{m+2}, s^{n}\} =\bar{c} s^{m+2}.
\end{align*}
where $c,\bar{c}$ are generic constants depending only on $n,s_0$. 
We have used that, by hypothesis, $m<n-1$ and that $\int_{s_0}^{+\infty} \xi ^{1+m} K_n(\xi)\, d\xi$ is bounded.

We proceed likewise with the behavior of $\mathcal{T}[\psi](s)$ as $s\to\infty$.
We take $s_1>0$ be such that the  expansions of the Bessel functions for $s\to \infty$ are true for $s>s_1$.
As  $\psi \in \BiO{l}{m}$, there exists  $\overline{C}$ be such that $\vert \psi(s) \vert \leq \overline{C}s^{-l}$ for $s>s_1$ and
$\vert \psi(s) \vert \leq\overline{C}$ for any $s\in J$.
We obtain
\begin{align*}
\vert \mathcal{T}[\psi](s)\vert \leq &
\overline{C} K_n(s) \int_{0}^s \xi I_n(\xi) \xi^{-l} \,d \xi + \bar C I_n(s) \int_{s}^{+\infty} \xi K_n(\xi)\xi^{l}\,d \xi \\
\leq & c e^{-s}s^{-1/2}\left(  \int_{0}^{s_1} \xi I_n(\xi)\,d\xi
+   \int_{s_1}^{s} \xi^{-l+1/2} e^{\xi}\,d\xi
+ \int_{s}^{+\infty}   \xi^{-l+1/2} e^{-\xi} \, d\xi \right)\\
\leq &\bar{c} s^{-l},
\end{align*}
where, as before, the values of $c,\bar{c}$ only depend on $n,s_1$. In conclusion $\mathcal{T}[\psi]\in \BiO{l}{m+2}$.
In particular, applying the above inequalities for $m=l=0$, that is $\psi$ bounded, we have that $\mathcal{T}[\psi]\in \BiO{0}{2}$. 

In addition, if $\psi$ is continuous, and since every integral in the
definition of $\mathcal{T}$ is uniformly convergent, we have that
\begin{equation*}
\mathcal{T}[\psi]'(s) = K_{n}'(s)\int_{0}^{s}\xi I_{n}(\xi) \psi(\xi)\, d\xi
+I_{n}'(s)\int_{s}^{\infty}\xi K_{n}(\xi)\psi(\xi)\, d\xi.
\end{equation*}
is also a continuous and bounded function.
We proceed as above to check the asymptotic expansions for $\mathcal{T}[\psi]'$.

\subsubsection{Proof of Lemma \ref{prop:f1}}
We notice that $\opf1_{\varphi}=\mathcal{T} \circ \mathcal{R}_{\varphi}$.
Let $g\in \mathcal{X}$.
It is clear that $\mathcal{R}_{\varphi}[g]\in \mathcal{X}$ and, by \eqref{BanachX}, 
$g, \mathcal{R}_{\varphi}[g]\in \BiO{3}{n-1}$. In consequence, the first item is a straightforward consequence of Lemma \ref{lemma:Top}.

To prove (ii) we need to show that there exists a constant $0<K<1$ such that, for any $g_1,g_2 \in \mathcal{X}$, 
$\Vert \opf1_{\varphi}[g_1]-\opf1_{\varphi}[g_2]\Vert_{\w} \leq K \Vert g_1-g_2\Vert_{\w}$. We first point out that, since
$0\leq \ft_0(s) \leq 1$ and by hypothesis (A2), $\partial_x^2(x\lambda(x))<0$, the function $DF(x)=x\lambda'(x)+\lambda(x)$ is decreasing. Therefore, using that by hypothesis (A1) $\lambda(1)=0$:
\begin{equation*}
-d=\lambda'(1)+\lambda(1)=DF(1)<DF(\ft_0(s))\leq DF(0)=\lambda(0)=1,
\end{equation*}
which gives
\begin{equation}
\label{derlambda}
0<\frac{DF(\ft_0(s))}{d}+1<\frac{1}{d}+1.
\end{equation}
Now we find that
\begin{equation*}
\begin{split}
\vert \opf1_{\varphi}[g_1](s)-\opf1_{\varphi}[g_2](s)\vert \leq  
&  K_{n}(s)\int_{0}^{s}\xi I_{n}(\xi)
\left(\frac{DF(\ft_0(\xi))}{d}+1\right)
\vert g_1(\xi)-g_2(\xi) \vert \, d\xi
\\&+I_{n}(s)\int_{s}^{\infty}\xi K_{n}(\xi)
\left(\frac{DF(\ft_0(\xi))}{d}+1\right)
\vert g_1(\xi)-g_2(\xi)\vert \, d\xi
\\&\leq  \Vert g_1-g_2\Vert_{\w} T(s) \,
\end{split}
\end{equation*}
where the function $T$ is defined by
\begin{equation}\label{defTnormal}
\begin{split}
T(s) :=  K_{n}(s)&\int_{0}^{s}\xi I_{n}(\xi) \left(\frac{DF(\ft_{0}(\xi))}{d}+1\right)\w(\xi)\, d\xi\\
&+I_{n}(s)\int_{s}^{\infty}\xi K_{n}(\xi) \left(\frac{DF(\ft_{0}(\xi))}{d}+1\right)\w(\xi)\, d\xi.
\end{split}
\end{equation}
We first observe that $T(s)>0$ $\forall s$ since both $K_n(s)$, $I_n(s)$  are positive, the weight function (see \eqref{defweigthfunction}) $\w(s)>0$
and inequality \eqref{derlambda}.

We now want to show that $\Vert T \Vert_{\w}<1$. We begin by rewriting $T$ in a more appropriate way.
Concretely, we will check that
\begin{equation}\label{defTalt}
\begin{split}
T(s)&= \w(s) -\mathcal{T}[h_0](s) \\
&=\w(s) - K_n(s) \int_{0}^s \xi I_n(\xi) h_0(\xi)\, d\xi
- I_n(s) \int_{s}^{+\infty} \xi K_n(\xi) h_0(\xi)\, d\xi
\end{split}
\end{equation}
being $\mathcal{T}$ the linear operator defined in Lemma \ref{lemma:Top} and
\begin{equation}\label{defh0}
h_0(s) = \frac{\sqrt{d}}{s^3}\left [2n^2 \ft_{0}(s) - \frac{s}{\sqrt{d}} \f{0}'\left (\frac{s}{\sqrt{d}}\right)\right ].
\end{equation}
To prove expression \eqref{defTalt} we deal with the differential equation that $\f{0}' (r)$ satisfies.
Indeed, since $\f{0}(r)$ is a solution of equation \eqref{f0}, $\f{0}'(r)$ is a solution of
the nonhomogeneous linear equation:
\begin{equation}\label{eqderf0}
\varphi''(r) + \frac{\varphi'(r)}{r} - \varphi(r) \frac{n^2}{r^2} +DF(\f{0}(r))\varphi(r)
= - \frac{2n^2}{r^3} \f{0}(r) + \frac{1}{r^2} \f{0}'(r).
\end{equation}
Performing the change $\psi(s) = \varphi(s/\sqrt{d})$ to this equation and taking into account that
$\ft_{0}(s) = \f{0}(s/\sqrt{d})$, we get that $\w(s) = \f{0}'(s/\sqrt{d})$ is a solution of
\begin{equation}\label{eqderf0s}
\psi''(s) + \frac{\psi'(s)}{s} - \psi(s) \frac{n^2}{s^2} + \frac{DF(\ft_{0}(s))}{d}\psi(s)
= - \frac{2n^2 \sqrt{d}}{s^3} \ft_{0}(s) + \frac{1}{s^2} \f{0}'(s/\sqrt{d}).
\end{equation}
We define
\begin{equation}\label{defLT}
\mathcal{L}[\psi](s) = \psi''(s) + \frac{\psi'(t)}{s} - \psi(s) \left (\frac{n^{2}}{s^{2}} + 1 \right).
\end{equation}
We notice that $\mathcal{L}[K_n]=\mathcal{L}[I_n]=0$
and that equation \eqref{eqderf0s}, for $\w$, can be rewritten as
\begin{equation}\label{eqderf0bis}
\left (\frac{DF(\ft_{0}(s))}{d}+1\right )\w(s) =
-\mathcal{L}[\w](s) -h_0(s) .
\end{equation}
The linear differential operator $\mathcal{L}$
satisfies that, upon integrating by parts, 
\begin{equation}\label{propL}
-\int_{a}^b \xi B_{n}(\xi) \mathcal{L}[\psi](\xi) \,d\xi
=- \left . \xi \psi'(\xi)B_n(\xi)\right |^{b}_{a} + \left . \xi \psi(\xi) B_n'(\xi)\right |^{b}_a,
\end{equation}
being either $B_n=K_n$ or $B_n=I_n$. This property was strongly used in \cite{Aguareles2011}.
Using that $\w$ satisfies equation \eqref{eqderf0bis}, property \eqref{propL} and that 
$s(I'_n(s) K_n(s)-K'_n(s)I_n(s) )=1$, we have that definition \eqref{defTnormal} of $T$ becomes
\begin{equation*}
\begin{split}
T(s)=&  -K_{n}(s)\int_{0}^{s}\xi I_{n}(\xi) \mathcal{L}[\w](\xi)\, d\xi-I_{n}(s)\int_{s}^{\infty}\xi K_{n}(\xi)
\mathcal{L}[\w](\xi) \, d\xi. \\
&- K_n(s) \int_{0}^s \xi I_n(\xi) h_0(\xi)\, d\xi
- I_n(s) \int_{s}^{+\infty} \xi K_n(\xi) h_0(\xi)\, d\xi \\
&= \w(s) - \mathcal{T}[h_0](s)
\end{split}
\end{equation*}
and \eqref{defTalt} is proven. 

Since by Proposition \ref{aefderf}, for any $r\ge 0$, $rf_0'(r)\le n^2f_0(r)$, we have that $h_0$, defined in \eqref{defh0}, satisfies that $h_0(s) >0$ 
if $s>0$, 
and therefore $\mathcal{T}[h_0](s)>0$, $s>0$. Consequently:
\begin{equation*}
0\le T(s) < \w(s) \;\;\; s>0.
\end{equation*}
Now, in order to check that $\Vert T \Vert_{\w}<1$ it only remains to see that
\begin{equation*}
\lim_{s\to 0} \frac{\mathcal{T}[h_0](s)}{\w(s)}\neq 0,\qquad \lim_{s\to +\infty} \frac{\mathcal{T}[h_0](s)}{\w(s)} \neq 0.
\end{equation*}
Indeed, recalling again definition \eqref{defweigthfunction} of $\w$, and using Proposition \ref{aefderf}, we have that
\begin{align*}
&\lim_{s\to 0} \w(s) s^{-n+1} = n\alpha d^{-(n-1)/2},\;\; \lim_{s\to +\infty} s^{3} \w(s)  = 2n^2 \sqrt{d},\\
\\ &\lim_{s\to 0} h_0(s)s^{-n+3}= \frac{n (2n-1)\alpha}{d^{(n-1)/2}}, \;\; \lim_{s\to +\infty} s^3 h_0(s) =2n^2 \sqrt{d}.
\end{align*}
Let $s_0>$ be small enough. By applying H\^opital's rule
\begin{align*}
\lim_{s\to 0} \frac{\mathcal{T}[h_0](s)}{\w(s)} &=\frac{d^{(n-1)/2}}{2n^2 \alpha} \lim_{s\to 0} \left( s^{-2n+1} \int_{0}^s \xi^{n+1} h_0(\xi)
+ s \int_{s}^{s_0} \xi^{-n+1} h_0(\xi)\right ) \\
&= \frac{d^{(n-1)/2}}{2n^2 \alpha} \lim_{s\to 0 } \left (\frac{s^{n+1} h_0(s)}{(2n-1) s^{2n-2}}+
s^{-n+3}h_0(s) \right)
\\&=\frac{1}{2n} + \frac{2n-1}{2n} =1.
\end{align*}
Now we deal with $s\to +\infty$. Let then $s_0>0$ be big enough. Then,
\begin{align*}
\lim_{s\to +\infty} \frac{\mathcal{T}[h_0](s)}{\w(s)}& = \frac{1}{4 n^2 \sqrt{d}} \lim_{s\to +\infty} \left (s^{{5/2}} e^{-s}\int_{s_0}^s e^{\xi} \xi^{1/2} h_0(\xi) \, d\xi
+ s^{5/2} e^{s} \int_{s}^{+\infty} e^{-\xi} \xi^{1/2} h_0(\xi)\, d \xi \right )\\
&=   \frac{1}{4 n^2 \sqrt{d}} \lim_{s\to +\infty} 2 s^3 h_0(s) =1.
\end{align*}

\textbf{Acknowledgements} M. Aguareles has been supported in part by grants from the Spanish Government MTM2011-27739-C04-03, MTM2014-52402-C3-3-P and and is a member of the Catalan research group 2014SGR1083. I. Baldom\'{a} and T.M-Seara have been partially supported by the Spanish
MINECO-FEDER Grant MTM2012-31714 and the Catalan Grant 2014SGR504 .
T. M-Seara has been partially supported by the Russian Scientific Foundation grant 14-41-00044 and Marie
Curie Action FP7-PEOPLE-2012-IRSES, BREUDS.

\bibliography{qneq0}

\bibliographystyle{alpha}
\end{document}